\documentclass{article}
\usepackage{amsmath}
\usepackage{amsthm}
\usepackage{amstext}
\usepackage[all,cmtip]{xy}
\usepackage{amssymb}
\usepackage[justification=centering]{caption}
 \usepackage{paralist}
  \usepackage{graphics} 
  \usepackage{epsfig} 
\usepackage{graphicx}  \usepackage{epstopdf}
 \usepackage[colorlinks=true]{hyperref}
 
\newtheorem{theorem}{Theorem}[section]
\newtheorem{lemma}[theorem]{Lemma}
\newtheorem{proposition}[theorem]{Proposition}
\theoremstyle{definition}
\newtheorem{definition}[theorem]{Definition}
\theoremstyle{remark}

\numberwithin{equation}{subsection}
\newtheorem{claim}[theorem]{claim}

\hypersetup{urlcolor=blue, citecolor=red}
\newcommand{\suppu}{\mathrm{supp}(\mu)}
\newcommand{\Wcs}{\mathcal{W}^{cs}}
\newcommand{\Wcu}{\mathcal{W}^{cu}}
\newcommand{\Wc}{\mathcal{W}^c}

\begin{document}

\title{Non-robustly intermingled basins for perturbations of time-one map of Anosov flows}

\author{Qianying Xiao\footnote{ Q. Xiao is the Corresponding author.} and Zuohuan Zheng }
\maketitle

{\footnotesize
 \centerline{Academy of Mathematics and Systems Science, Chinese Academy of Sciences}
   \centerline{}
   \centerline{Beijing, 100190, P. R. China}
} 

\bigskip

\begin{abstract}
Perturbations of time-one maps of transitive Anosov flows are studied. We show that most perturbations have no intermingled basins of hyperbolic physical measures.
\end{abstract}

\section{Introduction}
Let $M^d$ be a $C^\infty$ compact Riemannian manifold, $ f \in \mathrm{Diff}^2(M) $. To describe the behavior of almost all orbits is a goal of dynamical systems. It is meaningful to emphasize invariant measures which are physically observable. An invariant measure $\mu$ is a \emph{physical measure or SRB measure} if
\[ B(\mu)=\{ x \in M | \lim_{n \rightarrow +\infty } \frac{1}{n} \sum _{k=0}^{n-1} \delta_{f^k(x)} = \mu \}\]
has positive Lebesgue measure. J. Palis~\cite{Pa00,Pa05} conjectured that the existence and finiteness of physical measures whose basins cover Lebesgue almost every point hold for a generic map.

On the annulus $S^1\times [0,1]$, I. Kan~\cite{Kan} constructed a $C^2$ partially hyperbolic endomorphism
$$F:S^1\times [0,1]\rightarrow S^1\times [0,1]$$
such that $F(x,t)=(kx \mod 0,f_x(t)) $ with $k\geq 3$. $F$ is required to preserve the two boundaries $S^1\times \{0\}$ and $S^1\times \{1\}$, whose Lebesgue measures $m_1$ and $m_2$ are ergodic with negative center exponent and are physical measures consequently. Then one has $\mathrm{Leb}(B(m_1)\cup B(m_2))=1$. Even more interesting is that the Lebesgue density points of $B(m_i)$ are dense in $S^1\times [0,1]$ for $i=1,2$. Therefore $m_1$ and $m_2$ have \emph{intermingled basins}, by which we mean the closures of Lebesgue density points of $B(m_1)$ and $B(m_2)$ coincide, yet $B(m_1)$ and $B(m_2)$ are disjoint.

Uniformly hyperbolic systems have no intermingled basins. According to R. Bowen~\cite{Bow}, any physical measure is supported on an Axiom A attractor. The basin of the attractor coincides with the open attracting basin of the attractor modulo a zero Lebesgue measurable set.

One can easily see that the intermingled phenomenon discovered by Kan is $C^2$ robust. Y. Ilyashenko, V. Kleptsyn and P. Saltykov~\cite{Ily} show intermingled basins can even be $C^1$ robust for boundary preserving partially hyperbolic endormorphisms on the annulus.

Kan's example is adapted by R. Ures and C. V\'asquez in~\cite{Raul}, so that on certain boundaryless manifolds there exists partially hyperbolic diffeomorphisms with intermingled basins of physical measures.

In the same paper, Ures and V\'asquez show that for partially hyperbolic diffeomorphisms that are dynamically coherent on $\mathbb{T}^3$, intermingled basins of hyperbolic physical measures is not robust. They are reduced to the situation when every center leaf is compact. By studying the intersection of a u-saturated set with the center leaves, they relate intermingled basins to the existence of tori tangent to $ E^s\oplus E^u $.

One would ask if non-robustness of intermingled basins holds in other settings. In this work, we study perturbations of time-one map of transitive Anosov flows. We give a positive answer to the above question.

Let $M$ support a transitive Anosov flow $ \phi_t $ from now on. The time-one map $\phi_1$ is partially hyperbolic with one-dimensional center. That is to say there is a continuous splitting $TM=E^s\oplus E^c \oplus E^u$, $\dim E^c=1$, $E^s$ is uniformly contracted, $E^u$ is uniformly expanded, and the contraction and expansion of $E^c$ is dominated by $E^s$ and $E^u$ respectively.

According to Hirch-Pugh-Shub~\cite{HPS77}, there is a $C^1$ neighborhood $\mathcal{U}$ of $\phi_1$ in $\mathrm{Diff}^2(M)$ such that any $f\in\mathcal{U}$ is partially hyperbolic. There exist a foliation tangent to the center bundle. Moreover any $f\in\mathcal{U}$ is leafwise conjugate to $\phi_1$.

\begin{theorem}\label{theorem a}
There exists a $C^1$-open and $C^2$-dense subset $\mathcal{V} \subset \mathcal{U}$ such that for any $f \in \mathcal{V}$, any hyperbolic physical measures $\mu$ and $\nu$ of $f$ with disjoint supports do not have intermingled basins.

\end{theorem}

A partially hyperbolic diffeomorphism is accessible if any two points can be connected by piecewise differentiable curve with finite legs, each of which is tangent to the stable bundle $E^s$ or unstable bundle $E^u$. K. Burns, F. R. Hertz, J. R. Hertz, A. Talitskaya and R. Ures~\cite{Bur} prove that for partially hyperbolic diffeomorphisms with one-dimensional center, the collection having accessibility property is $C^r$ open and dense for $r\geq 1$.

By requiring $\mathcal{U}$ sufficiently small, the dynamics of $f\in \mathcal{U}$ at every center leaf is close to a translation by $1$ along the positive orientation. Therefore, on each center leaf $\Wc(x)$ the segment $[x,f(x)]_c$ connecting $x$ to $f(x)$ is generating in the following sense:
\[ \Wc(x)=\cup_{n\in\mathbb{Z}}f^n( [x,f(x)]_c ) .\]

The ideas of our arguments are inspired by~\cite{Raul}. We study the intersection of a u-saturated set with every center segment $[x,f(x)]_c$. Then we show intermingled basins of hyperbolic physical measures leads to a $C^1$-lamination tangent to $E^s\oplus E^u$, which contradicts accessibility. About lamination we will explain later.

We deal with hyperbolic physical measures in this paper. We do not know what happens to physical measures with zero center exponent. Moreover the supports of two physical measures is not necessarily disjoint. With the absence of dominated splitting, I. Melbourene and A. Windsor~\cite{MW} constructs on $\mathbb{T}^2$ a $C^\infty$ minimal diffeomorphism admitting any given number of absolutely continuous measures. To what extent partial hyperbolicity obstructs such intermingled supports is a problem. Moreover, the problem of intermingled basins of physical measures whose supports intersect is not considered in this paper.

This work is organized as following: Section \ref{pre} is the preliminaries. We will introduce the related properties of Gibbs u-states and physical measures in this section. In section \ref{the} we state our Proposition \ref{proposition b} from which Theorem \ref{theorem a} is deduced. Proof of Proposition \ref{proposition b} is in Section \ref{pro}.

\section{Preliminaries}\label{pre}

Let $f$ be a general partially hyperbolic diffeomorphism with one-dimensional center. The unstable foliation is denoted by $\mathcal{W}^{uu}$.

\begin{definition}
An invariant measure $\mu$ of $f$ is a \emph{Gibbs u-state} if the conditional measures of $\mu$ along $\mathcal{W}^{uu}$ are absolutely continuous with respect to the Lebesgue measures.
\end{definition}

Let us list some related facts about Gibbs u-states.
\begin{lemma}~\cite[Section 11.2]{BDV}\label{list}
\begin{enumerate}
\item A physical measure is a Gibbs u-state;
\item An ergodic Gibbs u-state with negative center exponent is a physical measure;
\item The ergodic components of a Gibbs u-state are Gibbs u-states;
\item The support of a Gibbs u-state is u-saturated.

\end{enumerate}
\end{lemma}

\begin{definition}
A point $x\in M^d$ is \emph{regular in the sense of Birkhoff} if
\[\lim_{n\rightarrow+\infty}\frac{1}{n}\sum^{n-1}_{k=0}\varphi(f^k(x))=\lim_{n\rightarrow -\infty}\frac{1}{-n}\sum^{n+1}_{k=0}\varphi(f^k(x)).\]
A point $x\in M$ is \emph{regular in the sense of Lyapunov} if there exists a splitting $T_xM=E_1(x)\oplus \cdots \oplus E_{k(x)}$, $k(x)\leq d$, Lyapunov exponents $\lambda_1(x)<\cdots<\lambda_k(x)$, such that for any $v\in E_i(x)\setminus\{0\}$, $i=1,\cdots,k(x)$,
\[ \lim_{n\rightarrow\pm\infty}\frac{1}{n}\ln\|Df^n(v)\|=\lambda_i(x).\]
\end{definition}

\begin{definition}
For any invariant measure $\mu$ of $f$, the integrated center exponent $\lambda^c(\mu)$ is defined by
\[\lambda^c(\mu)=\int\lambda^c(x)d\mu(x),  \]
with $\lambda^c(x)$ denotes the Lyapunov exponents along the one-dimensional center bundle.
\end{definition}

\begin{lemma}\label{main lemma}
If $\mu$ is a physical measure such that $\lambda^c(\mu)\neq 0$, then $\mu$ is ergodic.
\end{lemma}

\begin{proof}
 The proof is divided into two cases.
 \begin{enumerate}
 \item[Case 1]: $\lambda^c(\mu)<0$.

 Since $\mu$ is a Gibbs u-state, so is each of its ergodic component by Lemma \ref{list}. The ergodic component with negative center exponent is a physical measure. Hence $\mu$ has at most countably many ergodic components with negative center exponent. Let $\mu_1$ be such an component, the $\mu$ gives positive weight to $\mu_1$.

 According to Pesin theory, there exists $\Lambda\subset M\cap B(\mu_1)$ such that any $x\in \Lambda$ is regular in the sense of both Birkhoff and Lyapunov, the Pesin stable manifold $W^s(x)$ is well-defined, and $\mu_1(\Lambda)=1$. There are compact subsets $\Lambda_n$ such that $\Lambda_n\subset \Lambda_{n+1}$, $\Lambda=\cup_{n\geq1}\Lambda_n$, and the size of Pesin stable manifold of $x\in \Lambda_n$ is uniformly bounded away from 0.

 Take $\Lambda_N$ such that $\mu_1(\Lambda_N)>0$. There exists $\delta>0$ such that any $x\in \Lambda_N$, the size of $W^s$ is greater than $\delta$. Without loss of generality, we can assume each point of $\Lambda_N$ is in the support of $\mu_1|\Lambda_N$.

 Let $D$ be a disk in an unstable leaf such that $\mathrm{Leb}_D(D\cap\Lambda_N)>0$. Define
 \[R=\cup_{x\in D\cap\Lambda_N}W^s_{\delta/2}(x).\]
 Being in the support of $\mu_1|\Lambda_N$, $x$ is in the support of $\mu$. Together with the assumption that $\mu$ is a Gibbs u-state, one has $\mu(R)>0$.
 Take $y\in B(\mu)$ such that $y$ is a Lebesgue density point of $B(\mu)\cap\mathcal{W}^{uu}(y)$. The positive orbit of $y$ goes to arbitrarily close to $R$. By some bounded distortion arguments, one can easily show $B(\mu)\cap R\neq\emptyset$. Since $R\subset B(\mu_1)$, $\mu=\mu_1$.

 \item[Case 2]: $\lambda^c(\mu) >0$.

 Any $x\in B(\mu)$,
 \[\lim_{n\rightarrow+\infty}\frac{1}{n}\sum^{n-1}_{k=0}\ln \|Df^{-1}|E^c(f^k(x))\|=\int\ln \|Df^{-1}|E^c(x)\|d\mu(x)<0.\]
 Therefore $E^c\oplus E^u$ is mostly expanding in the sense defined by Alves-Bonatti-Viana~\cite{ABV}. Take a cu-disk D transverse to $\mathcal{W}^{ss}$ such that $\mathrm{Leb}_D(D\cap B(\mu))>0$. By ~\cite[Lemma 4.5]{ABV}, there exists physical measure $\nu$ whose basin is open and $D\cap B(\mu) \cap B(\nu)\neq\emptyset$. Consequently $\mu=\nu$.
 \end{enumerate}
Combining the two cases, we have shown that a physical measure of $f$ with nonzero integrated center exponent is ergodic.
\end{proof}

According to Lemma \ref{main lemma}, when we say a physical measure with negative (positive) center exponent, the meaning is clear, because the physical measure is ergodic, the center Lyapunov exponent equals the integrated center exponent almost everywhere.

\section{Proof of Theorem \ref{theorem a}}\label{the}

Let $\mathcal{U}$ be as in the Introduction. For any $f\in \mathcal{U}$, $E^s$ and $E^u$ is uniquely integrable, the integrated foliations are the stable foliation $\mathcal{W}^{ss}$ and the unstable foliation $\mathcal{W}^{uu}$. Moreover, $f$ is \emph{dynamically coherent}: there are foliations $\mathcal{W}^{cs}$ and $\mathcal{W}^{cu}$ tangent to $E^s\oplus E^c$ and $E^u\oplus E^c$, the intersection of $\mathcal{W}^{cs}$ and $\mathcal{W}^{cu}$ is a foliation $\mathcal{W}^c$.

A subset $K\subset M$ is \emph{u-saturated (s-saturated)} if it is the union of complete strong unstable leaves (stable leaves). \emph{c-saturated} set is defined similarly.

We will study certain u-saturated invariant compact set $K$ related to a physical measure in the next section. We will prove such kind of $K$ is laminated by $C^1$ leaves tangent to $E^s\oplus E^u$.

\begin{definition}~\cite{HPS77,Wil}
A \emph{lamination} $\mathcal{L}$ of a compact set $K\subset M$ is a family of disjoint submanifolds (leaves of the lamination) whose union is $K$ and which are assembeld in a $C^1$ continuous fashion. That is, $K$ is covered by lamination boxes, where a lamination box is a map $\varphi:D^c\times Y\rightarrow K$, $Y$ is a fixed compact set, $\varphi$ is a homeomorphism to a relatively open subset of $K$, and $\partial\varphi(x,y)/\partial x$ is nondegenerate and continuous with respect to $(x,y)\in D^c \times Y$.
\end{definition}

\begin{proposition}\label{proposition b}
For any $f\in \mathcal{U}$, any physical measure $\mu$ of $f$ with negative center Lyapunov exponent, if $K$ is a compact invariant u-saturated subset such that $K\subset \overline{B(\mu)}\setminus \suppu $, then
\begin{itemize}
\item there exists $k\in \mathbb{N}$ such that $K$ intersects each center leaf in exactly $k$ orbits,
\item $K$ is laminated by $C^1$ leaves tangent to $E^s\oplus E^u$. In particular, $f$ is not accessible.

\end{itemize}

\end{proposition}

\begin{definition}
Let $ \mu $, $ \nu $ be two physical measures. $ \mu $, $ \nu $ have \emph{intermingled basins} if for any open set $U$ in $M$, $ \mathrm{Leb}(U \cap B(\mu)) >0 $ implies $ \mathrm{Leb}(U \cap B(\nu)) >0 $ and vice versa.
\end{definition}
It is obvious from the definition that for two physical measures $ \mu $, $ \nu $ with intermingled basins, $\mathrm{supp}(\nu)\subset\overline{B(\mu)}$.

\begin{proof}[Proof of Theorem \ref{theorem a}]
According to ~\cite{Bur}, there is a $C^1$ open $C^2$ dense subset $\mathcal{V}\subset\mathcal{U}$ such that any $f\in\mathcal{V}$ is accessible.

For any $f\in\mathcal{V}$, any hyperbolic physical measures $\mu,\nu$ whose supports are disjoint but basins intermingle, one has $\lambda^c(\mu)<0,\lambda^c(\nu)<0$. Otherwise, by ~\cite{ABV}, $B(\mu)$ or $B(\nu)$ is open modulo a zero Lebesgue measurable set, the basins $B(\mu)$ and $B(\nu)$ do not intermingle.

Since $B(\mu)$ and $B(\nu)$ intermingle, $\mathrm{supp}(\nu)\subset \overline{B(\mu)}$. By assumption, $\suppu\cap \mathrm{supp}(\nu)=\emptyset$, hence $\mathrm{supp}(\nu)\subset \overline{B(\mu)}\setminus \suppu$.

Apply Proposition \ref{proposition b} to $K=\mathrm{supp}(\nu)$ and $\mu$, then $f$ is not accessible, a contradiction.

So any $f\in \mathcal{V}$ has no intermingled basins for hyperbolic physical measures with disjoint supports.
\end{proof}

\section{Proof of Proposition \ref{proposition b}}\label{pro}
Let $f\in \mathcal{U}$. Assume $K$ is a compact invariant u-saturated subset of $f$, and $K\subset \overline{B(\mu)}\setminus \suppu$.

\begin{lemma}\label{begin}
Every center unstable leaf is dense and $K$ intersects each center leaf.
\end{lemma}

\begin{proof}
Take $y\in M$ such that $\mathcal{W}^c(y)$ is dense(such $y$ exists since $f$ is leafwise conjugate to $\phi_1$ and $\phi_t$ is a transitive flow).

Any $x\in M$, $\mathcal{W}^{uu}(x)\cap \mathcal{W}^{cs}(y)\neq\emptyset$, let $x^\prime\in \mathcal{W}^{uu}(x)\cap \mathcal{W}^{cs}(y)\neq\emptyset$, then $\mathcal{W}^c(x^\prime)$ is dense, consequently $\mathcal{W}^{cu}(x)\supseteq \mathcal{W}^c(x^\prime)$ is dense.

Since $K$ is a compact invariant u-saturated subset, $\mathcal{W}^c(K)=\cup_{x\in K}\mathcal{W}^c(x)$ is both u-saturated and s-saturated. Moreover, $\mathcal{W}^c(K)$ is compact. Therefore $\mathcal{W}^c(K)=M$, i.e. $K$ intersects each center leaf.
\end{proof}

Let $\mathcal{K}(M)$ be the collection of nonempty subsets of $M$. Define
\begin{align*}
  \Phi_K: &\  M\rightarrow \mathcal{K}(M) \\
   & \ x\mapsto[x,f(x)]_c\cap K
\end{align*}
with $[x,f(x)]_c$ denote the closed segment on $\mathcal{W}^c(x)$ connecting $x$ and $f(x)$.

Since $\Wc(x)=\cup_{n\in \mathbb{Z}}f^n([x,f(x)]_c)$, $K$ is invariant, and $K\cap \Wc(x)\neq\emptyset$, one has $K\cap [x,f(x)]_c\neq\emptyset$. Therefore $\Phi_K$ is well-defined.

$K$ being compact, $\Phi_K$ is upper semi-continuous, and the collection of continuity points of $\Phi_K$ is a residual set of $M$.

Let $\mu$ be a physical measure with negative center exponent. By Lemma \ref{main lemma}, $\mu$ is ergodic. There exists $\Lambda\subset B(\mu)$ such that $\mu(\Lambda)=1$, $\Lambda$ consists of points that are regular in the sense of both Birkhoff and Lyapunov.
\begin{lemma}~\cite[Lemma 4.1, Lemma 4.3]{Raul}
\begin{description}
  \item[(1)] $\overline{B(\mu)}\subset\overline{W^s(\Lambda)}$, with $W^s(\Lambda)=\cup_{x\in\Lambda}W^s(x)$;
  \item[(2)] For any $x\in W^s(\Lambda)$, there exists $y\in \Wc(x)\cap \suppu$ such that $[x,y]_c\subset W^s(y)$.
\end{description}
\end{lemma}

\begin{lemma}~\cite[Lemma 4.4]{Raul}\label{h}
There exists $h>0$ such that any segment of length less than $h$ in a center leaf contains at most two points in $K$.
\end{lemma}
\begin{proof}
Assume for any $h>0$, there exist $x$, $y$, $z$ in a center segment of length less than $h$. Let $y$ lie between $x$ and $z$.

Let $k>0$ such that $\mathrm{dist}(K,\suppu)>k$. Since $y\in K\subset \overline{B(\mu)}\subset\overline{W^s(\Lambda)}$, there exist $q\in  W^s(\Lambda)$ arbitrarily close to $y$ and $p\in \suppu\cap\Wc(q)$ such that $[q,p]_c\subset W^s(p)$.

$h$ being small enough and $q$ close to $y$ as we like, $[q,p]_c$ must intersect either $\mathcal{W}^{ss}(\mathcal{W}^{uu}(x))$ or $\mathcal{W}^{ss}(\mathcal{W}^{uu}(z))$. Take a point $y^\prime$ from the intersection, then $\omega(y^\prime)\subset \suppu\cap K=\emptyset$, which is absurd.

So there exists $h>0$ such that any segment of length less than $h$ in a center leaf contains at most two points in $K$.
\end{proof}

\begin{lemma}
For any $x\in M$, $\sharp\Phi_K(x)<+\infty$.
\end{lemma}
\begin{proof}
It is obvious from Lemma \ref{h} and the fact that $[x,f(x)]_c$ is covered by finitely many segments with length less than $h$.
\end{proof}

Since $K$ is invariant and $\sharp\Phi_K(x)<\infty$ for any $x\in M$, there exists $k(x)\geq 1$ such that $\Wc(x)\cap K$ consists of $k(x)$ orbits of $f$, and $\sharp\Phi_K(y)=k(x)$ if $y$ is in $\Wc(x)\setminus K$, $\sharp\Phi_K(y)=k(x)+1$ if $y\in \Wc(x)\cap K$.

$k(x)$ is u-invariant because $K$ is u-saturated. Therefore \emph{$k(x)$ is constant on each center unstable leaf.
}
\begin{lemma}\label{const}
There is $k\in \mathbb{N}$ such that $k(z)=k$ for any $z\in M$.
\end{lemma}

\begin{proof}
Let $x$ be a continuity point of $\Phi_K$. There exists a neighborhood $U$ of $x$ such that $k(y)\geq k(x)$ for any $y\in U$.
\begin{claim}
$k(y)$ is constant in a small neighborhood $V$ of $x$.
\end{claim}
\begin{proof}[Proof of the Claim]
Assume there exist $\{y_n\} $ such that $y_n\rightarrow x$, $k(y_n)>k(x)$. Let $\bar{x}\in \Phi_K(x)$, $y_n^1,y_n^2\in \Phi_K(y_n)$ such that $y_n^i\rightarrow\bar{x},i=1,2$.

Let $n$ be sufficient large such that $y_n^1,y_n^2$ is close to $\bar{x}$ as we like and $\mathrm{d}(y_n^1,y_n^2)$ is small enough. $m$ is much greater than $n$ such that $\mathrm{d}(y_m^1,y_m^2)$ is much smaller than $\mathrm{d}(y_n^1,y_n^2)$.

Take $\omega_m^i$ from $\mathcal{W}^{ss}(\mathcal{W}^{uu}(y_n^i))\cap \Wc(y_m)$ for $i=1,2$. Then at most one of $\{\omega_m^1,\omega_m^2\}$ lies between $y_m^1$ and $y_m^2$. Applying similar arguments to Lemma \ref{h} can lead to a contradiction.
\end{proof}
For any $z\in M$, $\Wcu(z)$ is dense, therefore $\Wcu(z)\cap V\neq\emptyset$. Since $k(y)$ is constant along each center unstable leaf, $k(z)=k(x)$. So $k(z)$ is constant in $M$.
\end{proof}

Given a compact subset $\Lambda$ of $M$, a \emph{$r$ dimensional $C^0$ lamination} of $\Lambda$ is a partition $\mathcal{L}$ of $\Lambda$ such that: there is a collection of closed domains $\{D\}$, the union of whose interiors cover $\Lambda$, and for each domain $D$, there is a compact set $Y$ and a homeomorphism $\varphi: B^r\times Y\rightarrow D\cap \Lambda$ such that $\varphi(B^r\times{y})$ is the connected component of $\mathcal{L}(\varphi(0,y))\cap D$.

\begin{lemma}\label{c0}
$K$ has a codimension-one $C^0$ lamination $\mathcal{L}$.

\end{lemma}

\begin{proof}
Take $x$ outside $K$, $U$ a small neighborhood of $x$ such that $U\cap K=\emptyset$. Let $B$ a closed codimension-one disk contained in $U$ transverse to $\Wc$ and $x\in B$.

$D=\cup_{y\in B}[y,f(y)]_c$ is a closed domain. It is obvious that $K$ is covered by the interiors of such kind of $D$'s.

Any $y\in B$, $\Phi_K(y)=\{y_1,\cdots,y_k\}$ with $k$ as in Lemma \ref{const}. Suppose $y_i$'s are listed respecting the orientation of $\Wc$, and we can denote $y_1<y_2<\cdots<y_k$.

Define\begin{align*}
        \psi: \ B\times \{1,\cdots,k\} & \rightarrow K\cap D \\
        (y,i) & \mapsto y_i
      \end{align*}
By taking $B$ small enough we can see that $\psi$ is injective. We are left to show that $\psi$ is continuous.

Given any $(y,i)$, any $y_n\rightarrow y$, let $\bar{y_i}$ be an accumulation point of $\{(y_n)_i\}$. Then $\bar{y_i}=y_j$ for some $j$.

Suppose $j<i$, then there exist $1\leq j^\prime\leq j$, $1\leq i_1<i_2\leq i$ such that $(y_n)_{i_1},(y_n)_{i_2}$ accumulate to $y_{j^\prime}$. Similar arguments to Lemma \ref{h} will lead to a contradiction.

Neither $j>i$ holds.

Therefore $j=i$. $\psi$ is continuous at any given $(y,i)$.

Since $\psi$ is defined on compact space, $\psi$ is a homeomorphism.

$K$ is partitioned into connected components. The partition is denoted by $\mathcal{L}$.
The above arguments imply $\mathcal{L}$ is a codimension-one $C^0$ lamination.
\end{proof}

\begin{lemma}
$K$ is s-saturated.
\end{lemma}

\begin{proof}
Since if dynamically coherent, We are reduced to prove that for any $x\in K$, $\Wcs(x)\cap K$ is s-saturated. Since $E^s$ is uniquely integrable, we only need to show that $\Wcs(x)\cap K$ is tangent to $E^s(x)$.

Fix a small coordinate neighborhood $U$ of $x$ in $\Wcs(x)$ such that $x$ is the origin, $\Wc_{loc}(x)=\{0\}\times \mathbb{R}, \mathcal{W}^{ss}_{loc}(x)=\mathbb{R}^s\times\{0\}$.

Because $K$ is topologically transverse to $\Wc$ and by the arguments of Lemma \ref{c0} each leaf is an embedding, there exist $g:\mathbb{R}^s\rightarrow \mathbb{R}$ such that $g(0)=0$, $\mathrm{graph}(g)=\Wcs_{loc}(x)\cap K\cap U$.

Suppose $Dg(0)\neq0$, then there exist $v_n\in \mathbb{R}^s$ such that $\|v_n\|\rightarrow 0$, $|g(v_n)|/\|v_n\|\rightarrow 2b>0$.

Assume $g(v_n)>b\|v_n\|$ for some large $n$. For $0<a<\frac{b}{2}\|v_n\|$,
\[\mathcal{W}^{ss}_{loc}(0,a)=\{(x,y)|y=a+\varphi_a(x),\varphi_a(0)=0,\|D\varphi_a\|<\frac{b}{2}\}.\]

For $0\leq t\leq 1$, define $h_a(t)=g(tv_n)-\varphi_a(tv_n)-a$. Then
\begin{align*}
  h_a(1)=& \ g(v_n)-\varphi_a(v_n)-a\\
  > &\  b\|v_n\|-\frac{b}{2}\|v_n\|-\frac{b}{2}\|v_n\|\\
  > &\  0,
\end{align*}
\[h_a(0)=-a<0.\]
There exist $0<t<1$ such that $h_a(t)=0$. Therefore $g(tv_n)=\varphi_a(tv_n)+a$, that is to say $\mathrm{graph}(g)\cap \mathcal{W}^{ss}_{loc}(0,a)\ni(tv_n,g(tv_n))$. So $\mathrm{graph}(g)\cap \mathcal{W}^{ss}_{loc}(0,a)\neq\emptyset$ for $0<a<\frac{b}{2}\|v_n\|$.

 For a fixed $a$, let $p=(v_n,\varphi_a(v_n)+a)$, then $p\in \mathcal{W}^{ss}(0,a)\setminus K$.
There exist a neighborhoos $V$ of $p$ such that any $y\in V$, $\mathcal{W}^{ss}(y)\cap \mathrm{graph}(g)\neq \emptyset$.

Since $(tv_n,g(tv_n))\in \mathrm{graph}(g)\subset K\subset \overline{B(\mu)}\subset\overline{W^s(\Lambda)}$, there exist $q\in W^s(\Lambda) $ such that $q$ is close to $(tv_n,g(tv_n))$. $\mathcal{W}^{ss}(q)$ intersects $V$ at some $y$ by the continuity of strong stable foliation. Since $\mathcal{W}^{ss}(y)\cap \mathrm{graph}(g)\neq\emptyset$ one has $\mathcal{W}^{ss}(q)\cap \mathrm{graph}(g)\neq\emptyset$. Consequently $\omega(q)\subset K$.

On the other hand, $\omega(q)\subset \suppu$. so $\suppu\cap K\neq\emptyset$, a contradiction to the assumption.
Therefore $K$ is s-saturated.
\end{proof}

\begin{lemma}\label{end}
The leaves of $K$ are $C^1$ immersions.
\end{lemma}
\begin{proof}
By ~\cite[Appendix]{Ha}, $\mathcal{W}^{ss}(\mathcal{W}^{uu}(x))$ is once differentiable at $x$. Being both u-saturated and s-saturated,  every leaf of $K$ is $C^1$.
\end{proof}

\begin{proof}[Proof of Proposition \ref{proposition b}]
Let $k$ be as in Lemma \ref{const}. Combining  Lemma \ref{const} and Lemma \ref{end}, $K$ has a $C^0$ lamination $\mathcal{L}$ with $C^1$ leaves. The lamination being tangent to $E^s\oplus E^u$ which is continuous, is $C^1$. The proof of Proposition \ref{proposition b} is finished.
\end{proof}

\section*{Acknowledgement}
	
We would like to express our deep gratitude to Professor Raul Ures for posing the problem addressed in this paper and for for sharing his ideas generously.

Z. Zheng acknowledges support from NNSFC($\sharp11071238$), the Key Lab. of Random Complex Structures and Data Science CAS, and National Center for Mathematics and Interdisplinary Sciences CAS.


\end{document}